\documentclass[13pt]{article}  
\usepackage{amssymb}              
\usepackage{amsthm}
\usepackage{eucal}
\usepackage{verbatim}
\usepackage[dvips]{graphicx}
\usepackage{multirow}
\usepackage{fancyhdr}
\usepackage{color}
\usepackage{enumerate}
\usepackage{calrsfs}
\usepackage{fullpage}
\usepackage{amsmath}
\DeclareMathAlphabet{\pazocal}{OMS}{zplm}{m}{n}

\newtheorem{theorem}{Theorem}
\newtheorem{lemma}{Lemma}
\newtheorem{definition}{Definition}
\newtheorem{remark}{Remark}
\newtheorem{proposition}{Proposition}
\newtheorem{corollary}{Corollary}

\makeatletter
\newcommand{\leqnomode}{\tagsleft@true}
\newcommand{\reqnomode}{\tagsleft@false}
\makeatother

\def\({\begin{eqnarray}}
\def\){\end{eqnarray}}
\def\[{\begin{eqnarray*}}
\def\]{\end{eqnarray*}}
\def\part#1#2{\frac{\partial #1}{\partial #2}}

\def\A{\pazocal{A}}
\def\C{\pazocal{C}}
\def\H{\pazocal{H}}
\def\D{\pazocal{D}}
\def\I{\pazocal{I}}
\def\E{\pazocal{E}}
\def\S{\pazocal{S}}
\def\T{\pazocal{T}}

\begin{document}

\title{Invariance of velocity angles and flocking in the Inertial Spin model}   
\author{Ioannis Markou}  
\maketitle

\begin{abstract} We study the invariance of velocity angles and
flocking properties of the Inertial Spin model
introduced by Cavagna et al. [J. Stat. Phys., 158, (2015), 601--627]. We present a novel
approach, based on a second order nonlinear Gronwall
inequality for the velocity diameter, to identify invariant regions
and study the flocking behavior of the model. This is
an approach inspired by the work of Choi-Ha-Yun in the
study of synchronization for kuramoto oscillators with finite
inertia [Physica D, 240, (2011), 32--44]. We give sufficient conditions in terms of the
parameters of the model and the initial data, so that invariant regions and
asymptotic alignment of velocities is possible for coupling interactions
that satisfy general assumptions.
\end{abstract}


\textbf{2010 MR Subject Classification}: 82C22, 34D05, 34C15.

\section{Introduction to the Model} \label{Sec:Intro}
Collective behavior of self-propelled individual agents is observed
everywhere in nature, e.g. in swarming insects, herding of land animals, synchronization
in chemical oscillations etc. Many individual-based models (IBM models)
have been proposed to study otherwise complex phenomena such
as synchronization, flocking and swarming of birds, turning in schools of fish etc. In
the most successful models we should include the Kuramoto model for synchronization \cite{Ku, Wi}, the now classical three-zone model of Reynolds \cite{Re}, the Cucker-Smale model for flocking \cite{CuSm1,CuSm2}, the Vicsek
model for flocking \cite{ViCzBJCoSc}, models that study turning in schools of fish \cite{DeMo1,DeMo2},
other orientation based flocking models e.g. \cite{DeFrMe, HaKiLeNo} etc. The research in the above otherwise simple models has followed so many different paths, that it would be futile to mention the entire list. For instance, we may direct the reader to following literature e.g. in Vicsek type models \cite{ChGiGrPeRa, CrMePaBr, Gi, ToTu}, delayed Cucker-Smale systems \cite{ChHas1, ErHaSu, HaMa1, HaMa2}, Cucker-Smale with asymmetric weights \cite{MoTa}, synchronization with inertial effects \cite{ChHaMo, ChHaYu1, ChLiHaXuYu}, or the survey articles \cite{BeHaOu, BeDeTa, ViZa} etc.

In this paper we study emergence behavior in a relatively new model called
the Inertial Spin (IS) model that was introduced to
describe the mechanism in which collective turns occur in groups of birds.
The IS model is a model for the collective
dynamics of particles that takes into account, not only position and velocity of particles, but also
the inertial angular momentum (spin) of each agent. The model was introduced in \cite{AtCaCaGiJeMi, CaCaGiGr} and
was later studied in \cite{BeBuCa, HaKiKiSh, YaMa}. Its introduction and justification came to treat one major flaw
of the classical Vicsek model \cite{ViCzBJCoSc}. Indeed, the classical Vicsek model describes the
evolution of self-propelled particles in 2-d having a constant speed, where the direction of the
velocity is updated by computing the average of neighboring directions.
One of the drawbacks of the Vicsek model is the inadequacy to describe how collective turns occur. In experiments
that have been performed in real flocks it has been observed that when a bird starts to turn, then the
turning information is propagated along the group and the whole group performs a collective turn. This has been
addressed with the introduction of an extra (spin) variable that accounts for the turning information. In the
new system each particle is described by vectors $x_i, v_i, s_i \in \mathbb{R}^3$ that stand for the position, velocity and spin of the $i'$ th particle.

We describe the Inertial Spin model by a system of ODEs for the positions, velocities, and spins, i.e.
\begin{align} \label{IS} \left\{ \begin{array}{ll} \dot{x}_{i}(t)&=v_{i}(t), \quad i=1
\ldots, N, \quad t>0 ,
\\ \\ \dot{v}_{i}(t)&=\frac{1}{\chi} s_i(t) \times v_i(t) ,
\\ \\ \dot{s}_i(t)&=v_i(t) \times \left[\frac{k}{N}\sum \limits_{j=1}^N
\psi_{ij}(v_j(t) -v_i(t)) -\gamma \dot{v}_i(t) \right],
\end{array} \right.
\end{align}
with prescribed initial conditions that satisfy:
\begin{equation} \label{IC} (x_i,v_i,s_i)\Big|_{t=0}=(x_{i0},v_{i0},s_{i0}), \quad s_{i0} \cdot v_{i0}=1, \,
\text{and normalized initial speed} \quad |v_{i0}|=1, \quad 1\leq i \leq N. \end{equation}
Here the constant $\chi>0$ is the moment of inertia, $\gamma>0$ is the
friction (damping) parameter, and $k>0$ is the coupling strength of interactions. The
matrix $(\psi_{ij})_{i,j}$ is the matrix of communication weights between particles $i$ and $j$
which measures the strength of the social interaction between the two agents.
Since all the vectors are in $\mathbb{R}^3$ we denote by $\cdot$ and $\times$ the
usual inner (dot) and outer (cross) product between two vectors. Similarly, by $|\cdot|$ we denote
the induced euclidian distance in $\mathbb{R}^3$. We also consider the vectors of positions, velocities and spins defined by
\begin{equation} \label{Def} x:=(x_1,\ldots, x_N)\in \mathbb{R}^{3N}, \qquad v:=(v_1,\ldots, v_N)\in \mathbb{R}^{3N}, \qquad s:=(s_1, \ldots,s_N)\in \mathbb{R}^{3N}. \end{equation}
The derivation of the IS system is beyond the scope of this paper and can be found in e.g.
\cite{CaCaGiGr, HaKiKiSh}.

The final component of system \eqref{IS} that we address is the communication matrix $(\psi_{ij})_{ij}$.
The $\psi_{ij}$ element of the communication matrix describes the social interaction
between the $i$ and $j$ agents. The minimum of assumptions that we make for $\psi_{ij}$ is that the $(\psi_{ij})_{ij}$ matrix
is nonnegative, symmetric i.e. $\psi_{ij}=\psi_{ji}$ for all pairs $i,j$, and its elements are bounded by some $\psi_M>0$.
There are two main frameworks: The first one is a general framework that considers variable in time weights, with only basic assumption a lower bound on the
$\psi_{ij}$ values, i.e that \begin{equation} \label{PsiAssum1}\psi_m \leq \psi_{ij}(t) \leq \psi_M , \qquad \text{for some}\quad \psi_M > \psi_m > 0 , \quad \forall i,j \qquad \forall t\geq 0. \end{equation} This assumption is considerably more general than the case of constant communication weights, and is very useful in systems with variable social interactions.
Another realistic framework is the case where $\psi_{ij}:=\psi(|x_i-x_j|)$ is
a function of the metric distance between the $i,j$ particles. The function $\psi: [0, \infty) \to \mathbb{R}_+$ in this framework is bounded, decreasing, and Lipschitz continuous, \begin{align} \label{PsiAssum2} & 0 < \psi(r) \leq \psi(0), \qquad
[\psi]_{Lip}:=\sup \limits_{x \neq y}\frac{|\psi(x)-\psi(y)|}{|x-y|}<\infty ,
\\  \nonumber &(\psi(r_2)-\psi(r_1))(r_2-r_1) \leq 0, \quad \forall r,r_1,r_2>0 .\end{align} A classical
example in this category is the communication rate in the
classical Cucker-Smale model \cite{CuSm1, CuSm2}, where $\psi(r)$
takes the form $\psi(r)=(1+r^2)^{-\beta/2}$ for some $\beta>0$. In this article we will work mainly with the
first framework, where communication weights are allowed to vary with time but have a strong positive bound from below.

The rest of this article is structured as follows. In Section \ref{Sec:Prelim}
we give some of the most important properties of the IS model, and then we present the most important
existing result on its aggregate behavior. After that, we present the main results of this paper. In the short Section \ref{Sec:Simpl} that follows, we show how the IS model is connected to other systems of collective behavior. In Section \ref{Sec:Veloc-diam}, we derive the basic differential inequality for the velocity diameter. In Section \ref{Sec:Gron}, we give important lemmas that allow the treatment of the inequalities we derived in Section \ref{Sec:Veloc-diam}. Finally, in Section \ref{Sec:Proofs} we prove the main results.

\section{Preliminaries and Main results} \label{Sec:Prelim}

\subsection{Preliminaries and Definitions}
The Inertial Spin model enjoys some important properties (e.g. conservations, invariances etc) which
we try to include in this section, starting with the following Proposition from
\cite{HaKiKiSh} (see Propositions 2.1 \& 2.2).

\begin{proposition} \label{Prop:Properties}  \cite{HaKiKiSh}
Assume that $(x(t),v(t),s(t))$ is the solution to system \eqref{IS}, with initial
conditions $(x_0,v_0,s_0)$ that satisfy \eqref{IC}. Then, the following assertions hold:
\\ (i) The speed of each particle is conserved,
\begin{equation*} \frac{d}{dt}|v_i(t)|=0 , \quad t>0, \quad i=1,\ldots, N. \end{equation*}
\\ (ii) The inner product of the spin and velocity is conserved,
\begin{equation*}\frac{d}{dt}(s_i(t) \cdot v_i(t))=0 , \quad t>0, \quad i=1,\ldots, N. \end{equation*}
\\ (iii) The average spin defined by $s_c(t):=\frac{1}{N} \sum \limits_{i=1}^N s_i(t)$ dissipates with rate $\frac{\gamma}{\chi}$, i.e.
\begin{equation*} s_c(t)=s_c(0) e^{-\frac{\gamma}{\chi}t} .\end{equation*}
\\ (iv) System \eqref{IS} has rotational symmetry, i.e. for an orthogonal $3 \times 3$ matrix $O \in O(3)$ and
the new dynamical variables $\widetilde{x}_i, \widetilde{v}_i, \widetilde{s}_i$ defined by
\begin{equation*} \widetilde{x}_i:=O x_i, \qquad \widetilde{v}_i:=O v_i,
\qquad \widetilde{s}_i:=Os_i ,\end{equation*} then
$(\widetilde{x}, \widetilde{v}, \widetilde{s})$ is also a solution of \eqref{IS}.
\end{proposition}

Simple geometric identities offer some further insights on the system \eqref{IS}. For a start, the
identity $v_i \times (s_i \times v_i)=(v_i \cdot v_i)s_i -(s_i \cdot v_i)v_i=s_i$ means that we can rewrite the spin equation $(1)_3$ as
\begin{equation} \label{Sp} \dot{s}_i(t)=\frac{k}{N}\sum \limits_{j=1}^N \psi_{ij}
v_i(t) \times v_j(t)  -\frac{\gamma}{\chi} s_i(t) .\end{equation}
We may derive a second order equation for $\ddot{v}_i$ using the system \eqref{IS}. Indeed, differentiating $(1)_2$ and then using $(1)_3$ we get
\begin{align} \nonumber \chi \ddot{v}_i &=\dot{s}_i \times v_i +s_i \times \dot{v}_i \\ \nonumber &=
\left[ v_i \times \left( \frac{k}{N}\sum \limits_{j=1}^N \psi_{ij}v_j -\gamma \dot{v}_i \right)\right]
\times v_i +\frac{1}{\chi} s_i \times (s_i \times v_i) \\ \label{IS2} &=\frac{k}{N}\sum \limits_{j=1}^N \psi_{ij}
(v_i \times v_j)\times v_i -\gamma (v_i \times \dot{v}_i)\times v_i +\frac{1}{\chi}s_i \times (s_i \times v_i). \end{align}
We now use the following identities in \eqref{IS2} (all of which follow from the triple product formula
$a \times (b \times c)=b(a \cdot c)-c(a \cdot b)$ for $a,b,c
\in \mathbb{R}^3$) \begin{align*} (v_i \times v_j)\times v_i
&=v_j -(v_i \cdot v_j)v_i ,  \qquad (v_i \times \dot{v}_i)\times v_i=(v_i \cdot v_i)
\dot{v}_i-(v_i \cdot \dot{v}_i)v_i=\dot{v}_i , \\ s_i \times (s_i \times v_i)&=(s_i \cdot v_i)s_i-
(s_i \cdot s_i)v_i=-|s_i|^2 v_i=-\chi^2 |\dot{v}_i|^2 v_i ,\end{align*}
and we end up with the following second order equation for velocities
\begin{equation} \label{IS1} \chi \ddot{v}_i + \gamma \dot{v}_i +\chi |\dot{v}_i|^2 v_i =\frac{k}{N}\sum
\limits_{k=1}^{N}\psi_{ik}(v_k -(v_i \cdot v_k)v_i) .\end{equation} The formulation of \eqref{IS1} allows us to
have a first glimpse of the mechanism that drives the velocity alignment in certain configurations. Indeed, in this
formulation the RHS can be written as
\begin{equation*} \frac{k}{N} \sum \limits_{j=1}^N \psi_{ij}
\Gamma(v_i,v_j) , \qquad \text{where} \quad \Gamma(v_i,v_j)=v_j- (v_i \cdot v_j)v_i \end{equation*}
is the orthogonal projection of vector $v_j$ on $v_i$. The term $\Gamma(v_i,v_j)$ is a nonlinear and non-symmetric
coupling term that plays a role on preserving the speed of the agents due to the property
$\Gamma(v_i,v_j) \cdot v_i=0$. This equation reduces to the Cucker-Smale model with constant speed when we take $\chi \to 0$, i.e.
\begin{equation} \label{C-S_unit} \dot{x}_i=v_i, \qquad \dot{v}_i=\frac{\bar{k}}{N}\sum \limits_{j=1}^N \psi_{ij}\Gamma(v_i,v_j), \qquad \text{where}\quad \bar{k}=\frac{k}{\gamma}. \end{equation}

This version of CS model with constant speed has been studied in \cite{ChHa1, ChHa2, HaKoZh}.
The idea of taking a coupling term $\Gamma(v_i,v_j)$ which takes the orthogonal projection
of $v_j$ on $v_i$ instead of $v_j-v_i$ goes back to a model of quantum synchronization
due to Lohe \cite{Lo1}. We should also note that other instances of C-S dynamics with a
non linear coupling $\Gamma(v_i, v_j)$ have been studied in the literature.
For instance, in \cite{HaHaKi, Ma} the authors studied a velocity coupling of the
type $\Gamma(v_i, v_j)=(v_j-v_i)|v_j-v_i|^{2(\gamma -1)}$ for a
constant $\gamma \in \left( \frac{1}{2},\frac{3}{2}\right)$, and gave conditions for flocking as well as collision avoidance.

For a start we define the diameters of positions, velocities, and spins in the usual way
\begin{equation*} D(x(t)) :=\max\limits_{1\leq i,j \leq N}|x_i(t) -x_j(t)|, \quad
D(v(t)) :=\max\limits_{1\leq i,j \leq N}|v_i(t) -v_j(t)|,
\quad D(s(t)) :=\max\limits_{1\leq i,j \leq N}|s_i(t) -s_j(t)| .\end{equation*}
Our point of view is to study the evolution of velocity diameters by means of studying
the maximal angles between velocities since these two are geometrically connected.
For this, if we assume that all velocities $v_i$ lie in a cone with opening (apex) angle $\pi/2$, then we
can study the evolution of angle between $v_i$ and $v_j$ by studying the inner product of these two velocities.
In this situation, by taking the minimum of the inner product of all velocity pairs, we can control the velocity diameter of the system.

Thus, given the considerations above, we introduce the
geometric factor of motion $\A(v)$ just as in \cite{ChHa1}, which was defined as
\begin{equation} \label{geom_factor} \A(v(t)):=\min \limits_{i \neq j} v_i(t) \cdot v_j(t). \end{equation}
When all the velocities lie inside a cone with angle in $\pi/2$, then $\A(v)>0$ and we can have some control of the velocity diameter. As a result, the geometric factor
plays a crucial role in identifying invariance velocity regions. Furthermore,
the pair of velocities that is the minimizer in the factor
$\A(v)$ is the same pair that maximizes the velocity diameter with
the exact relationship between these two quantities
being $\A(v(t))=1-\frac{1}{2}D(v(t))^2$ when $\A(v(t))>0$.
We now proceed with some definitions:
\begin{definition} Let $(x(t),v(t),s(t))$ be a smooth solution to
system \eqref{IS}, with initial data $(x_0,v_0,s_0)$ that satisfy conditions \eqref{IC}. The following
definitions hold:
\begin{itemize}
\item (Spin vanishing) The solution exhibits asymptotic vanishing of spins if and only if
\begin{equation*} D(s(t)) \to 0 \quad \text{as} \quad t \to \infty ,\end{equation*}
which due to average spin dissipation implies $s_i(t) \to 0$ for all $1\leq i \leq N$.
\item (Orientation invariance) We say that the solution exhibits orientational invariance if and only if
all velocity angles remain inside a cone with slant angle less than $\pi /2$
for all times. In other words, if there exists some $0<\beta <1$ such that
$\A (v(t))>\beta$ for all $t \geq 0$.
\item (Velocity alignment) The solution exhibits velocity alignment
if and only if $D(v(t))\to 0$ as $t \to \infty$.
\item (Asymptotic flocking) We say that the solution exhibits asymptotic flocking if
and only if we have velocity alignment and group coherence i.e.
\begin{equation*} D(v(t))\to 0 \quad \text{as} \quad t \to \infty, \qquad \exists \quad
0< D^\infty <\infty \quad \text{s.t.} \quad \sup \limits_{t\geq 0}D(x(t))<D^\infty .
\end{equation*} If this holds true for all initial data $(x_0,v_0,s_0)$ we speak
of \textsl{unconditional flocking}, otherwise if flocking is possible only for a
specific subset of initial conditions, then we speak of \textsl{conditional flocking}.
\end{itemize}
\end{definition}

In this article we want to address the following question:
\begin{itemize}
\item \textbf{Q1} : \emph{Under what assumptions in
terms of the parameters $\chi, \gamma, k$ of the problem, the communication matrix $\psi_{ij}$, and in terms
of the initial conditions, we can establish invariant
regions? Also, what are the assumptions we need to show velocity alignment or even emergence of flocks?}
\par
A more specific question we ask is the following:
\item \textbf{Q2} : \emph{Does the smallness of initial diameters of the state variables suffice for the proof of orientation invariance for the velocities and/or flocking, or do we need further assumptions on the parameters $k, \chi, \gamma$ in the model?}
\end{itemize}
\textbf{Notation:} For brevity we might have to suppress the dependence of any state variable from time $t$, especially when we are treating longer
formulas e.g. we may write $\dot{s}_i$ or $D(v)^2$ instead of
$\dot{s}_i(t)$ and $D(v(t))^2$ etc.

\subsection{State of the art and main Theorems}

We should start with the state of art results in the study of asymptotics for the IS model. So far,
the only known result is the one presented in \cite{HaKiKiSh} for weights that are multiplicative, i.e.
for weights that satisfy $\psi_{ij}=p_i p_j$ for any $1 \leq i,j \leq N$, and a given set $\{ p_i\}_{1 \leq i \leq N}$ of positive elements $p_i>0$. This result
was given in two important steps. In the first step, it is shown that for constant
weights (which is a more general assumption than that of multiplicative weights) we can deduce
asymptotic spin vanishing as is shown in the following:

\begin{proposition} \cite{HaKiKiSh} \label{Prop:Ha} Suppose
that parameters $k, \gamma, \chi$ are all positive. If $(x(t),v(t),s(t))$ is a solution to the
IS model with communication weights that are nonnegative constant and symmetric, then asymptotic spin vanishing occurs. In particular, we have that
\begin{equation*} \lim \limits_{t \to \infty}|s_i(t)|=0 \quad \text{and} \quad \lim \limits_{t \to \infty}|\dot{s}_i(t)|=0 .\end{equation*}
\end{proposition}
\begin{proof}
A sketch of the main steps in the proof of Proposition \ref{Prop:Ha} is possible in one paragraph. First, we introduce the following
functionals \begin{equation*} \E(t):=\frac{1}{N^2}\sum \limits_{i,j}
\psi_{ij}|v_i-v_j|^2, \qquad \S(t):=\frac{1}{N}\sum \limits_{i}|s_i|^2 . \end{equation*}
We can show with a direct computation that $\frac{\chi}{2}\E(t)+\frac{1}{k}\S(t)$ is dissipative for symmetric, and constant
communication weights by showing that
\begin{equation*}\frac{d}{dt}\left( \frac{\chi}{2}\E(t)+\frac{1}{k}\S(t) \right)+\frac{2 \gamma}{\chi k}\S(t)=0, \qquad t>0.
\end{equation*} This in turn implies the a priori energy estimate
\begin{equation*} \frac{\chi}{2}\E(t)+\frac{1}{k}\S(t) +\frac{2\gamma}{\chi k}\int_0^t \S(s)
\, ds \leq \frac{\chi}{2}\E(0)+\frac{1}{k}\S(0), \end{equation*}
or equivalently \begin{equation} \label{S-integr} \int_0^{\infty} \S(t)\, dt \leq \frac{\chi k}{2 \gamma}
 \left( \frac{\chi}{2} \E(0)+\frac{1}{k} \S(0) \right) .\end{equation} The next step is
to show that $\dot{\S}(t)$ is uniformly bounded, and as a result $\S(t)$ is uniformly continuous. Then,
with the help of Barbalat's lemma \cite{Bar} the integrability of $\S(t)$ together with its uniform continuity implies that $\S(t)\to 0$ as $t \to \infty$. Similarly, it can be shown that $\dot{\S}(t)$ is also uniformly continuous and since $\S(t) \to 0$ it also follows that $\dot{\S}(t) \to 0$. \end{proof}

We now impose the additional assumption of multiplicativity in the communication weights.
Then, with the help of Proposition \ref{Prop:Ha} and assumptions on the smallness of initial data $\E(0), \S(0)$,
it was proven in \cite{HaKiKiSh} that the IS system reaches a 2-flock or 1-flock formation asymptotically in time.
For this it helps to introduce the following averaging quantities
\begin{equation*} p_c:=\frac{1}{N}\sum \limits_{i=1}^N p_i \qquad \bar{v}_c:=\frac{1}{N}\sum \limits_{i=1}^N p_i v_i .\end{equation*}
The following main result was proven in \cite{HaKiKiSh} (see Theorem 4.1):

\begin{proposition} \cite{HaKiKiSh} \label{Theorem:Ha}
In addition to the assumptions made in Proposition \ref{Prop:Ha} we assume that
the communication weights are also multiplicative, i.e. that $\psi_{ij}=p_i p_j$ where $p_i>0$ for any
$1 \leq i,j \leq N$. Then the following statements hold :
\\ (i) If the initial data $(x_0,v_0,s_0)$ satisfies the condition
\begin{equation*} \E(0)+\frac{2}{\chi k}\S(0) <2 p_c^2,
\end{equation*}
then each velocity $v_i$ converges to either $\bar{v}_c /|\bar{v}_c |$ or $-\bar{v}_c /|\bar{v}_c |$.
\\ (ii)  If the initial data $(x_0,v_0,s_0)$ satisfies the more restrictive condition
\begin{equation*} \E(0)+\frac{2}{\chi k}\S(0)< \min \limits_{1 \leq i \leq N}
\left\{ \frac{8p_i(Np_c -p_i)}{N^2}\right\}, \end{equation*}
then we have velocity alignment and all velocities converge to $\bar{v}_c /|\bar{v}_c |$.
\end{proposition}

\begin{proof} The proof is presented in \cite{HaKiKiSh}.
\end{proof}

There are some remarks regarding Proposition \ref{Theorem:Ha} that we have to address before we go on with the presentation of the results of this work :

\begin{remark}
It is important to note that Proposition \ref{Theorem:Ha} applies only to a very restrictive class of
weights that adhere to the multiplicativity property. As it was pointed in \cite{HaKiKiSh}, this property appears to be relevant in fragmentation-coagulation processes of polymers, but it is still a very restrictive framework to work with in the general problem. In fact, even the class of constant communication weights is not wide enough to use in many modelling examples e.g. in cases where the interaction weights depend on the metric distance between agents, or
when the network topology changes in time etc.
\end{remark}

\begin{remark}
The convergence that is proven in Proposition \ref{Theorem:Ha} does not possess any explicit rate. Hence it is impossible to deduce coherence in group formation from this result alone, i.e. we don't have sufficient information to show flocking in the traditional way that was defined earlier. In many physically realistic cases we would like to have a way to show that when alignment in velocities happens, the group of particles remains bounded uniformly in time, or
at least, we have the vanishing of the velocity diameter in some explicit rate.
\end{remark}

\begin{remark}
The IS model has a natural tendency to induce orientation alignment in the particle group when the velocities and spins are not that far apart. This property of the
IS system becomes evident when we write the system as a second order equation for the velocities of agents. Unfortunately, the functional approach that is the cornerstone of Proposition \ref{Theorem:Ha} does not make any use of the structure that is ``build in'' the IS system. In this paper, our approach is to make full use of this structure.
\end{remark}

Having these considerations in mind, we are in position to state the main two results of this work.
First, we state a theorem that works for positive, constant and symmetric communication weights (not necessarily multiplicative). This theorem states that in such a case we have invariant regions for the velocity vectors and we can achieve velocity alignment in these regions but with no explicit rate.

\begin{theorem} \label{Theorem1} Let $(x(t),v(t),s(t))$ be a global smooth solution to
\eqref{IS}, with a symmetric and constant communication matrix $\{\psi_{ij}\}_{i,j=1}^N$,
and initial data $(x_0,v_0,s_0)$ that satisfy conditions \eqref{IC}. Suppose also that
$\psi_m=\min \limits_{1\leq i,j \leq N}
\psi_{ij}$, and $\psi_M=\max \limits_{1\leq i,j \leq N} \psi_{ij}$. Finally, we assume some $\delta_0 \in (0,1)$ s.t.
 the following two conditions hold: \begin{equation*} \A(v_0)>\delta_0 \qquad \text{and} \qquad \C_0 <1- \delta_0 , \end{equation*} where
the exact form of $\C_0:=\C(D(v_0),\dot{D}(v_0),D(s_0),s_0,\gamma,k,\chi)$ is given by the following formulas :
\\(i) If $\gamma > \sqrt{8k \chi \psi_m \delta_0 }: $
\begin{equation*} \C_0 :=\frac{1}{2} D(v_0)^2 +\frac{\chi}{\gamma \sqrt{1- 8k \chi \psi_m \delta_0 / \gamma^2}}
\left( D(v_0)|\dot{D}(v_0)|+\frac{1}{2} \nu_1 D(v_0)^2 \right) + \frac{2 k N}{\gamma \sqrt{2 k \chi \psi_m \delta_0}} \left( \frac{\chi}{2}\E(0)+\frac{1}{k} \S(0)\right)\end{equation*}
\\(ii) If $\gamma \leq \sqrt{8k \chi \psi_m \delta_0 }:$
\begin{align*} \C_0:= \left\{ \begin{array}{ll}  & \frac{1}{2}D(v_0)^2 \left( 1+\frac{2\chi}{\gamma} \right)
+ \frac{4 k N}{\gamma^2}\left( \frac{\chi}{2}\E(0)+\frac{1}{k} \S(0)\right), \qquad \text{for}\quad \dot{D}(v_0) \geq 0 \\ \\
  & \frac{1}{2}D(v_0)^2 + \frac{4 k N}{\gamma^2}\left( \frac{\chi}{2}\E(0)+\frac{1}{k} \S(0)\right), \qquad \text{for}\quad \dot{D}(v_0) < 0 .\end{array} \right. \end{align*}
Then the particle system has the property $\A(v(t))>\delta_0$
for all $t\geq 0$, and the velocity diameters vanish asymptotically and $D(v)\to 0$ as $t \to \infty$.
\end{theorem}

The second result that we prove is a stronger result that can be achieved in the ``Small Inertia'' regime, and
gives exponentially fast alignment of velocities under certain assumptions on the initial configurations.

\begin{theorem} \label{Theorem2} Let $(x(t),v(t),s(t))$ be a global smooth solution to the IS system \eqref{IS},
with initial data $(x_0,v_0,s_0)$ that satisfy \eqref{IC}. We have also assumed that the communication matrix in \eqref{IS}
satisfies $\psi_m \leq \psi_{ij}(t) \leq \psi_M$ for all $1 \leq i,j \leq N$ and $t \geq 0$,
and that $\A(v_0)>\delta_0$ for some $\delta_0 \in (0,1)$.
Finally, we assume that the initial configurations and parameters satisfy the conditions \begin{equation*} (H1)
\quad \gamma  > \sqrt{\frac{6 k \chi \psi_M}{\delta_0 \psi_m} } ,\qquad \qquad (H2)
\quad \C_0 < 1- \delta_0 ,\end{equation*}
where the exact form of $\C_0$ is given by
\begin{equation*} \C_0:= \frac{1}{2}D(v_0)^2
+\frac{\chi D(v_0) |\dot{D}(v_0)|}{\gamma}+\frac{2}{\gamma ^2}\left(D(s_0)^2+
2\max \limits_{1 \leq i \leq N}|s_{i0}|^2 \right) .\end{equation*} Then the particle flow exhibits orientation invariance
with $\A(v(t))>\delta_0$ for $t\geq 0$. Furthermore, we have the estimates
\begin{align} \label{D(v)_estim} D(v(t))\leq \left\{ \begin{array}{ll}    O(e^{-\frac{\mu_{*}}{2}t}), \qquad \text{for}\quad \mu_* < \frac{\gamma}{2\chi}, \\ \\
 O(\sqrt{t}e^{-\frac{\gamma}{4 \chi}t}), \qquad \text{for}\quad \mu_* \geq \frac{\gamma}{2\chi} , \end{array} \right. \end{align}
where $\mu_*=\sup \D$ for the non-empty set
\begin{equation} \label{D_set} \D=\left\{0< \mu <\frac{\gamma}{\chi} \quad : \quad \mu^3 -\frac{2\gamma}{\chi}\mu^2 +\frac{2k \chi \psi_m \delta_0 +\gamma^2}{\chi^2}\mu \leq \frac{2k}{\chi}\left( \frac{\gamma}{\chi}\psi_m \delta_0 -\frac{6k}{\gamma} \psi_M \right) \right\} . \end{equation}
\end{theorem}

\begin{remark} The way we have defined $\C_0$ in both theorems ensures that the
 $\C_0<1-\delta_0$ assumption is stronger than $\A(v_0)>0$, and hence $\A(v_0)>\delta_0$ could have been omitted. We have kept it because it signifies the importance of the geometric factor in the analysis, and it is also helpful, since we work with a continuity argument on $\A(v(t))$ to show that $\A(v(t))>\delta_0$ for all $t\geq 0$ in both proofs.
\end{remark}

\section{Some special cases of IS models} \label{Sec:Simpl}

In this brief Section we discuss some results for two
simplified instances of the IS model: Namely, the IS model on 2-d plane where it reduces to the inertial Kuramoto model, and the overdamped regime $\chi \to 0$, where the IS reduces to the Cucker-Smale with constant speed (for weights that are functions of the metric distance of particles).

\subsection{Inertial Spin on the plane and connection to the inertial Kuramoto model}

If we assume a motion restricted on a plane then the velocity vector $v_i$ can be expressed in polar form as
\begin{equation*} v_i=(\cos{\theta_i}, \sin{\theta_i}) ,\end{equation*}
for $\theta_i \in [0,2\pi)$. We Differentiate the velocity vector to get the first and second derivatives
\begin{equation*} \dot{v}_i=(-\sin{\theta_i}, \cos{\theta_i})\dot{\theta}_i  \qquad \text{and} \qquad
 \ddot{v}_i=(-\cos{\theta_i},-\sin{\theta_i})|\dot{\theta}_i|^2 +(-\sin{\theta_i},\cos{\theta_i})\ddot{\theta}_i.\end{equation*}
Substituting the derivatives of $v_i$ into equation \eqref{IS1}, then using
 the identity $\chi |\dot{v}_i|^2 v_i=\chi (\cos{\theta_i},\sin{\theta_i})|\dot{\theta}_i|^2$
and the fact that \begin{equation*} \sum \limits_{k=1}^N \psi_{ik} (v_k -(v_i \cdot v_k)v_i)= (-\sin{\theta_i},\cos{\theta})
\sum \limits_{k=1}^N \psi_{ik} \sin{(\theta_k -\theta_i)} ,\end{equation*}
we get the equation for the evolution of $\theta_i$, i.e.
\begin{equation} \label{Kuram_frict} \chi \ddot{\theta_i}+\gamma \dot{\theta_i}=\frac{k}{N}\sum \limits_{k=1}^N
\psi_{ik}\sin{(\theta_k -\theta_i)}, \quad 1 \leq i \leq N .\end{equation}
This is the Kuramoto model equation when inertial effects are involved, for a network topology where $\psi_{ij}:=a_{ij}$ is the coupling strength between oscillators $i$ and $j$, and with zero natural frequencies $\Omega_i=0$. This model is coupled with the set of initial conditions
\begin{equation} \label{Kuram_IC} \theta_i(0)=\theta_{i0}, \quad
\dot{\theta}_i(0)= \dot{\theta}_{i0}, \qquad 1 \leq i \leq N. \end{equation}
The inertial Kuramoto model has been studied extensively in many works e.g. \cite{ChHaJuSl, ChHaMo, ChHaNo, ChHaYu1, ChHaYu2, ChLiHaXuYu, HaLi}. In particular, we may consider the more general model
\begin{equation} \label{Kuram_In} m_i \ddot{\theta}_i+\gamma_i \dot{\theta}_i=\Omega_i +k \sum \limits_{j=1}^N a_{ij}\sin(\theta_j-\theta_i) , \quad t \geq 0, \quad i=1,\ldots,N ,\end{equation}
with prescribed initial data given by \eqref{Kuram_IC}. Here $m_i , \gamma_i$ are the inertia and damping parameters of the i'th oscillator respectively, $k$ is the coupling strength, and the $\Omega_i$'s are the natural frequencies of the oscillators which are independent to one another and each
one follows a common distribution $g(\Omega)$. The unknown function $\theta_i(t)$ corresponds to
the phase of the i'th oscillator, and its derivative $\omega_i(t):=\dot{\theta}_i(t)$ to its frequency.
We use the shorthand vector notation $(\theta(t),\omega(t)):=(\theta_1,\ldots,\theta_N,\omega_1,\ldots,\omega_N)$ and the initial data are represented by the vector $(\theta_0, \omega_0)$. We also define the diameters for the phases and frequencies by
\begin{equation*}  D_{\theta}(t):=\max \limits_{1\leq i \leq N}|\theta_j -\theta_i|, \qquad  D_{\omega}(t):=\max \limits_{1\leq i \leq N}|\omega_j -\omega_i| .\end{equation*}
We say that the system exhibits \textit{complete phase synchronization} if $D_{\theta}(t)\to 0$, and \textit{complete frequency synchronization} if $D_{\omega}(t)\to 0$ as $t \to \infty$.

The following result regarding the case of identical natural
frequencies, and homogeneous inertia and friction, was the
first result of its kind for \eqref{Kuram_In} that was proven in \cite{ChHaYu1}. Instead of using the formulation
of the theorem from \cite{ChHaYu1} (where the aggregate result is presented in two Theorems 4.1 \& 4.2), we choose the
more synoptic formulation in \cite{ChHaMo}:

\begin{proposition} \label{Theorem:ChHaYu1} \cite{ChHaMo}
Suppose that the parameters $m_i, \gamma_i, \Omega_i, k$ and initial data satisfy the following relations:
\begin{align*} & m_i=m, \quad \gamma_i=1, \quad a_{ij}=\frac{1}{N}, \quad \Omega_i=\Omega , \quad m, k>0 \\
& 0< \C_1(0) < \pi \qquad \text{and} \qquad mk \in \left(0,\frac{1}{4}\right)\cup \left(\frac{\C_1(0)}{4 \sin{\C_1(0)}},\infty \right), \end{align*}
where $\C_1(m,\theta(t),\dot{\theta}(t)):=\max\{ D_{\theta}(t), \, D_{\theta}(t)+m \dot{D}_{\theta}(t)\}$. Then,
for any solution $(\theta(t), \omega(t))$ to \eqref{Kuram_In}, there exists some positive constant $\Lambda_1>0$ s.t.
\begin{equation*} D_{\theta}(t)+D_{\omega}(t)\leq O(1)e^{-\Lambda_1 t} \quad \text{for}\quad t \to \infty.\end{equation*}

\end{proposition}

\begin{remark}
The observant reader might have already noticed the similarity between the above result and Theorem \ref{Theorem2} of this article. It is important for instance to see that the initial configurations in Proposition \ref{Theorem:ChHaYu1} (as manifested by the assumption $0<\C_1(m,\theta_0, \omega_0)<\pi$) are chosen so that all initial phases are contained on the semi-circle. This implies a certain orientation that is assumed in initial conditions, and it is a common assumption in both Proposition \ref{Theorem:ChHaYu1} and Theorem \ref{Theorem2} (in Theorem \ref{Theorem2} this assumption is $\C_0 < 1- \delta_0$). On the other hand, there is also a major difference between the two results. Although Theorem \ref{Theorem2} makes the assumption of a ``small inertia'' regime with the condition $\gamma > \sqrt{\frac{6 k \chi \psi_M}{\delta_0 \psi_m} }$, Proposition \ref{Theorem:ChHaYu1} applies in both small and large inertia regions, as manifested by $mk \in \left(0,\frac{1}{4}\right)\cup \left(\frac{\C_1(0)}{4 \sin{\C_1(0)}},\infty \right)$. This shouldn't come as a surprise, as the 2-d setting of the IS system allows for a much sharper inequality for the phase diameter, than the one we derive in Section \ref{Sec:Veloc-diam} for the velocity diameter of the IS in 3-d. The problem of showing invariant regions for the velocities in the large inertia regime, when the communication matrix is not constant, remains open.
\end{remark}

\subsection{Formal zero inertial limit $\chi \to 0$}

In the overdamped regime, when the inertia becomes
negligible compared to the spin dissipation and $\psi_{ij}=\psi(|x_i-x_j|)$, the IS
model reduces to the CS model with constant speed. By taking the formal limit
$\chi \to 0$ (derivation could be made rigorous!) we get system \eqref{C-S_unit}, which we rewrite as
\begin{equation} \label{C-S_unit1}\dot{x}_i=v_i, \qquad \dot{v}_i=\frac{\bar{k}}{N}\sum \limits_{j=1}^N \psi_{ij}(v_j-(v_i \cdot v_j)v_i), \qquad \text{with}\quad \bar{k}=\frac{k}{\gamma}. \end{equation}
This model has been studied in \cite{ChHa1} and later in \cite{ChHa2, HaKoZh}.
In general, it is easy to construct initial configurations for a simple three particle system where
a global flock is not achieved (see \textbf{Appendix. A} in \cite{ChHa1}). Furthermore, despite the simplicity of this system, the asymptotic behavior is not yet fully understood for all initial configurations. That being said,
for the restricted class of initial conditions s.t. $\A(v_0)>0$ it is possible to employ a Lyapunov functional
approach to prove asymptotic flocking. Indeed, if we consider initial data with $\A(v_0)>0$ then we can prove the following system of dissipating differential inequalities (SDDI) holds
\begin{equation} \label{SDDI}\Big| \frac{d}{dt}D(x)\Big| \leq D(v),
\qquad \frac{d}{dt}D(v)\leq -\bar{k}\psi(D(x)) D(v) \A(v_0) ,\end{equation}
which directly implies that $\A(v(t)) \geq \A(v_0)$ for all $t \geq 0$.
From this point, the standard treatment is to introduce the Lyapunov functional
\begin{equation} \label{C-S_unitLyap} \H_{\pm}(t):=D(v)\pm \bar{k}\A(v_0) \int_{0}^{D(x)}
\psi(s)\, ds , \end{equation}
and show with the help of \eqref{SDDI} that this functional dissipates,
i.e. $\H_\pm (t)\leq \H_\pm (0)$ for all $t \geq 0$. Then the following flocking estimate for \eqref{C-S_unit1} was
proven.
\begin{proposition} \label{Theorem:ChHa1} \cite{ChHa1} Suppose that the communication weight $\psi(\cdot)$ and
initial configurations $(x_0,v_0)$ satisfy the following conditions:
\begin{equation*} \A(v_0)>0, \qquad 0 < D(v_0)< \bar{k}\A(v_0)
\min{\left\{ \int_{D(x_0)}^\infty \psi(s)\, ds , \int_{0}^{D(x_0)} \psi(s)\, ds \right\}} .\end{equation*}
Then, there exists a unique solution $(x(t),v(t))$ to system \eqref{C-S_unit} satisfying the
asymptotic flocking conditions:
\begin{equation*}  \sup \limits_{t\geq 0}D(x) < D^{\infty} \qquad \text{and}
\qquad  \, D(v)\leq D(v_0) e^{-\bar{k}\A(v_0) \psi(D^{\infty}) t}, \quad t\geq 0, \end{equation*}
where $D^\infty$ is defined implicitly by $D(v_0):=\bar{k}\A(v_0) \int_{D(x_0)}^{D^\infty}\psi(s)\, ds$.
\end{proposition}

We should also note here that invariant regions of velocities and the flocking property have
also been studied in a ``cone-vision'' variant of the C-S model in \cite{LiHuWu}. In the
``cone-vision'' model the agents update their velocities by only considering other agents that belong in a certain vision-cone of their velocity vector. In general,
such a model can give birth to multi-cluster formations, but conditions for the emergence of a unique flock in the 2-d and 3-d models were presented in \cite{LiHuWu}.

\section{A second order differential inequality for the velocity diameter} \label{Sec:Veloc-diam}

We now prove the main inequality for velocity diameters. In our proof we assume the more general communication
$\psi:=\psi(|x_j-x_i|)$ where $\psi(\cdot)$ is a decreasing weight function s.t. $0<\psi(\cdot)<\psi_M $. We
make no assumption of a lower bound for $\psi(\cdot)\geq \psi_m>0$ in this.

\begin{proposition} \label{Prop:Vel-diam} Let $(x,v,s)$ be a solution to \eqref{IS}
with initial data $(x_0,v_0,s_0)$ that satisfy \eqref{IC}. We also assume
that this solution has all flight directions (orientations of velocities)
chosen s.t. $\A(v(t))>0$ for all $t\geq 0$. Then we have
\begin{align}  \label{BasIneq1}  \frac{d^2}{dt^2} D(v)^2
+ \frac{\gamma}{\chi} \frac{d}{dt}D(v)^2 +  \frac{2 k}{\chi}\psi(D(x)) D(v)^2 \A(v)
\leq \frac{4}{\chi^2}\left( D(s)^2 + \max \limits_{1\leq i \leq N}|s_i|^2 D(v)^2 \right). \end{align}
Moreover, the inequality for $D(v)$ is
\begin{align}  \label{BasIneq2}  \frac{d^2}{dt^2} D(v)^2
+ \frac{\gamma}{\chi} \frac{d}{dt}D(v)^2 +  \frac{2 k}{\chi}\psi(D(x)) D(v)^2 \A(v)
\leq \C_1 \int_0^t ds \, e^{-\frac{\gamma}{\chi}(t-s)}D(v(s))^2  +\C_2 e^{-\frac{\gamma}{\chi}t} , \end{align}
where \begin{equation*} \C_1=\frac{4k^2 \psi_M}{\gamma \chi}(1+D(v)^2)< \frac{12 k^2 \psi_M}{\gamma \chi}\quad \text{and}\quad  \C_2= \frac{4}{\chi^2}\left( D(s_0)^2 +2\max \limits_{1\leq i \leq N}|s_i(0)|^2\right)\end{equation*}

\end{proposition}

\begin{proof}
We write down eq. \eqref{IS1} for two particles $i,j$ and take the inner product
with $v_j$ and $v_i$ respectively. The idea is that by obtaining an equation for the inner product
$v_i \cdot v_j$ we can then derive an equation for $|v_i-v_j|^2$.
\begin{equation*} \chi \ddot{v}_i \cdot v_j +\gamma \dot{v}_i \cdot v_j
+\chi |\dot{v}_i|^2 v_i \cdot v_j = \frac{k}{N}\sum
\limits_{k=1}^{N}\psi_{ik}(v_k \cdot v_j -(v_i \cdot v_k)v_i \cdot v_j) ,\end{equation*}
\begin{equation*} \chi \ddot{v}_j \cdot v_i +\gamma \dot{v}_j \cdot v_i
+\chi |\dot{v}_j|^2 v_i \cdot v_j = \frac{k}{N}\sum
\limits_{k=1}^{N}\psi_{jk}(v_k \cdot v_i -(v_j \cdot v_k)v_j \cdot v_i) .\end{equation*}
Now if we sum both equations we get
\begin{align} \nonumber & \chi (\ddot{v}_i \cdot v_{j} +\ddot{v}_i \cdot v_{j})
+\gamma \frac{d}{dt} (v_i \cdot v_j) + \chi (|\dot{v}_i|^2 +|\dot{v}_j|^2) v_i \cdot v_j= \\
 \label{Inn} & \frac{k}{N}\sum
\limits_{k=1}^{N}\psi_{ik}(v_k \cdot v_j -(v_k \cdot v_i) v_i \cdot v_j)+
\frac{k}{N}\sum
\limits_{k=1}^{N}\psi_{jk}(v_k \cdot v_i -(v_k \cdot v_j) v_i \cdot v_j ).\end{align}
Note that we have $\frac{d^2}{dt^2} (v_i \cdot v_j)=\ddot{v}_i \cdot v_j
+2\dot{v}_i \cdot \dot{v}_j +\ddot{v}_j \cdot v_i$, so eq. \eqref{Inn} yields
\begin{align} \nonumber & \chi \frac{d^2}{dt^2} (v_i \cdot v_j)
+\gamma \frac{d}{dt} (v_i \cdot v_j) - 2 \chi \dot{v}_i \cdot \dot{v}_j+ \chi (|\dot{v}_i|^2
+|\dot{v}_j|^2)v_i \cdot v_j= \\ \label{MaxAng1} & \frac{k}{N}\sum
\limits_{k=1}^{N}\psi_{ik}(v_k \cdot v_j -(v_k \cdot v_i)  v_i \cdot v_j)+
\frac{k}{N}\sum
\limits_{k=1}^{N}\psi_{jk}(v_k \cdot v_i -(v_k \cdot v_j) v_i \cdot v_j )\end{align}
We now use the identity $v_i \cdot v_j=1-\frac{1}{2}|v_i-v_j|^2$ in \eqref{MaxAng1} to get
\begin{align} \nonumber & \frac{\chi}{2} \frac{d^2}{dt^2} |v_i-v_j|^2
+\frac{\gamma}{2} \frac{d}{dt}|v_i-v_j|^2 = - 2 \chi \dot{v}_i \cdot \dot{v}_j + \chi (|\dot{v}_i|^2
+|\dot{v}_j|^2)v_i \cdot v_j \\ \label{MaxAng2}&
-\frac{k}{N}\sum
\limits_{k=1}^{N}\psi_{ik}(v_k \cdot v_j -(v_k \cdot v_i)  v_i \cdot v_j)-
\frac{k}{N}\sum
\limits_{k=1}^{N}\psi_{jk}(v_k \cdot v_i -(v_k \cdot v_j) v_i \cdot v_j )
 .\end{align}
Note also that since $- 2 \chi \dot{v}_i \cdot \dot{v}_j+ \chi (|\dot{v}_i|^2
+|\dot{v}_j|^2)v_i \cdot v_j \leq \chi |\dot{v}_i
-\dot{v}_j|^2$ inequality \eqref{MaxAng2} yields
\begin{align}  \label{MaxAng3} \frac{\chi}{2} \frac{d^2}{dt^2} |v_i -v_j|^2
+\frac{\gamma}{2} \frac{d}{dt}|v_i -v_j|^2 \leq \I_1 +\I_2, \end{align}
where $\I_1:=\chi |\dot{v}_i
- \dot{v}_j|^2$ and
\begin{align*} \I_2:=\I_{21}+\I_{22}:=-\frac{k}{N}\sum
\limits_{k=1}^{N}\psi_{ik}(v_k \cdot v_j -(v_k \cdot v_i)  v_i \cdot v_j)-
\frac{k}{N}\sum
\limits_{k=1}^{N}\psi_{jk}(v_k \cdot v_i -(v_k \cdot v_j) v_i \cdot v_j ) .\end{align*}
To control the term $\I_1$ we first write $\dot{v}_i-\dot{v}_j$ as
\begin{equation*} \dot{v}_i-\dot{v}_j= \frac{1}{\chi}(s_i -s_j)\times v_i +\frac{1}{\chi}s_j \times (v_i -v_j) ,\end{equation*}
by virtue of \eqref{IS}. We therefore have that
\begin{equation*} \I_1 \leq \frac{2}{\chi}|s_i -s_j|^2
+\frac{2}{\chi}|s_j|^2 |v_i-v_j|^2 ,\end{equation*}
which yields \begin{equation} \label{I1} \I_1 \leq \frac{2}{\chi}D(s)^2 +\frac{2}{\chi}\max \limits_{1 \leq i \leq N}|s_i|^2 D(v)^2 .\end{equation}

We proceed with the treatment of $\I_2$. In this treatment we make the
assumption that the $i,j$ indices are chosen so that $v_i$ and $v_j$ are the
extremal velocities that minimize $v_i \cdot v_j$, i.e $v_j \cdot v_k \geq v_i \cdot v_j$ for
all $1\leq k \leq N$. Indeed, with this choice for the $i,j$ pair we have that
\begin{align*} \I_{21} &=-\frac{k}{N}\sum \limits_{k=1}^N \psi_{ik} (v_k \cdot v_j
-(v_k \cdot v_i) v_i \cdot v_j) \\ & \leq -\frac{k}{N} \min \limits_{i,j}\psi_{ij}\sum \limits_{k=1}^N (v_k \cdot v_j
-(v_k \cdot v_i) v_i \cdot v_j) \\ &\leq -\frac{k}{N} \psi(D(x))\sum \limits_{k=1}^N (v_k \cdot v_j -(v_k \cdot v_i) v_i \cdot v_j).
\end{align*}
Note that we used the fact that $v_k \cdot v_j -(v_k \cdot v_i) v_i \cdot v_j\geq 0$
(since $v_k \cdot v_j \geq v_i \cdot v_j$ and $|v_k \cdot v_i|\leq 1$). In similar manner we show for $\I_{22}$ that
\begin{align*} \I_{22} &=-\frac{k}{N}\sum \limits_{k=1}^N \psi_{jk} (v_k \cdot v_i
-(v_k \cdot v_j) v_i \cdot v_j) \\ & \leq -\frac{k}{N} \min \limits_{i,j}\psi_{ij}\sum \limits_{k=1}^N (v_k \cdot v_i
-(v_k \cdot v_j) v_i \cdot v_j) \\ &\leq -\frac{k}{N} \psi(D(x))\sum \limits_{k=1}^N (v_k \cdot v_i -(v_k \cdot v_j) v_i \cdot v_j).
\end{align*}
If we add the two expressions after we re-arrange terms we have
\begin{align*} \I_2 &\leq -\frac{k}{N} \psi(D(x))\sum \limits_{k=1}^N (v_k \cdot v_j -(v_k \cdot v_j)
v_i \cdot v_j) -\frac{k}{N} \psi(D(x))\sum \limits_{k=1}^N (v_k \cdot v_i -(v_k \cdot v_i) v_i \cdot v_j) \\
& \leq -\frac{k}{N} \psi(D(x))\sum \limits_{k=1}^N (1 - v_i \cdot v_j) v_k \cdot v_j-\frac{k}{N}
\psi(D(x))\sum \limits_{k=1}^N (1 - v_i \cdot v_j) v_k \cdot v_i
\\ &\leq -\frac{k}{2N} \psi(D(x)) D(v)^2 \sum \limits_{k=1}^N v_k \cdot v_j
-\frac{k}{2N} \psi(D(x)) D(v)^2 \sum \limits_{k=1}^N v_k \cdot v_i \leq -k
\psi(D(x))D(v)^2 v_i \cdot v_j .
\end{align*}
This yields the following bound for $\I_2$
\begin{equation}\label{I2} \I_2 \leq - k\psi(D(x)) D(v)^2 \A(v) . \end{equation}
Inserting the estimates \eqref{I1}-\eqref{I2} in \eqref{MaxAng3}, and keeping in mind that we have chosen $i, j$ s.t. $D(v)=|v_i -v_j|$, we get the first inequality \eqref{BasIneq1}.

Finally, we turn to the term $\I_1$ which controls the RHS of \eqref{BasIneq1}.
To control $ |s_i-s_j|^2$ we use the spin equation \eqref{Sp}
for $s_i-s_j$, and after we multiply by $s_i-s_j$ we get
\begin{equation*} \frac{d}{dt}|s_i-s_j|^2 +\frac{2\gamma}{\chi} |s_i-s_j|^2
=\frac{2k}{N}\sum \limits_k \left((\psi_{ik} v_i -\psi_{jk}v_j)
\times v_k \right) \cdot (s_i-s_j) .\end{equation*}
We note that when $0\leq \psi(\cdot)\leq \psi_M$ then $|\psi_{ik} v_i -\psi_{jk}v_j| \leq \psi_M |v_i -v_j|$. Now using
the Young inequality we get
\begin{align*} \frac{d}{dt}|s_i-s_j|^2 +\frac{2\gamma}{\chi} |s_i-s_j|^2
 \leq \frac{k^2 \chi \psi_M}{\gamma} |v_i -v_j|^2 +\frac{\gamma}{\chi}|s_i-s_j|^2 .\end{align*}
With the help of Duhamel's principle we get
\begin{align*} |s_i-s_j|^2 \leq \frac{k^2 \chi \psi_M}{\gamma} \int_0^t ds \, e^{-\frac{\gamma}{\chi}
(t-s)}|v_i(s)-v_j(s)|^2+ e^{-\frac{\gamma}{\chi} t} |s_{i0}-s_{j0}|^2 .\end{align*}
Thus, picking $i,j$ s.t. $D(s)=|s_i-s_j|$ we have
\begin{equation*} D(s)^2 \leq \frac{k^2 \chi \psi_M}{\gamma}
\int_0^t ds \, e^{-\frac{\gamma}{\chi}(t-s)} D(v(s))^2 +
 e^{-\frac{\gamma}{\chi}t}D(s_0)^2. \end{equation*}
A similar bound can be obtained for $|s_i|^2$, i.e.
\begin{equation*} \frac{d}{dt}|s_i|^2 +\frac{\gamma}{\chi} |s_i|^2 \leq \frac{k^2 \chi \psi_M}{\gamma N}
\sum \limits_{k=1}^N |v_i \times v_k|^2.\end{equation*}

Now using the fact that $|v_i \times v_j|^2=|v_i-v_j|^2-\frac{1}{4}|v_i-v_j|^4$,
which implies $|v_i \times v_j|^2 \leq |v_i-v_j|^2$, we get the estimate
\begin{equation*} |s_i|^2 \leq e^{-\frac{\gamma}{\chi}t} |s_{i0}|^2+\frac{k^2 \chi \psi_M}{\gamma}
\int_0^t ds \, e^{-\frac{\gamma}{\chi}(t-s)}D(v(s))^2 .\end{equation*}
In this case, collecting the terms for $\I_1$ we get
\begin{align*}  \I_1 & \leq \frac{2 k^2 \psi_M}{\gamma} \int_0^t ds \, e^{-\frac{\gamma}{\chi}
(t-s)}D(v(s))^2+ \frac{2}{\chi} e^{-\frac{\gamma}{\chi} t} D(s_0)^2
 \\  &+ \frac{2 k^2 \psi_M}{\gamma} D(v)^2 \int_0^t ds \, e^{-\frac{\gamma}{\chi} (t-s)}
D(v(s))^2 +\frac{2}{\chi} e^{-\frac{\gamma}{\chi} t}  D(v)^2
 \max \limits_{1\leq i \leq N}|s_{i0}|^2.\end{align*}
\end{proof}

\begin{remark}
 In the derivation of inequalities \eqref{BasIneq1} -- \eqref{BasIneq2} we have used
the fact that the flight orientations are all contained in an invariant region $\A(v(t))>0  \, \,
\text{for} \, \, t \geq 0$. The purpose of this work is to
 show that if $\A(v_0)>\delta_0$ (for any choice $\delta_0 \in (0,1)$),
then under certain assumptions on the initial data
and parameters of the model, and with the help of a continuity
argument on \eqref{BasIneq1} -- \eqref{BasIneq2}, we have that $\A(v(t))>\delta_0$ for all $t\geq 0$.
\end{remark}

\begin{remark}
In addition to identifying invariant regions with the help of Proposition \ref{Prop:Vel-diam} we want to know
whether convergence to a common velocity can be derived from \eqref{BasIneq1} -- \eqref{BasIneq2}. Some of the obvious
difficulties in treating \eqref{BasIneq1} -- \eqref{BasIneq2} are: (i) the nonlinearity
in the term $\psi(D(x))D(v)^2 \A(v)$, (ii) the non-homogeneity in the RHS, and (iii) the fact that they are both
second order differential inequalities. All of the above imply the
necessity to work with new types of Gronwall inequalities. In the Section that follows, we derive such inequalities.
\end{remark}

\section{Second order Gronwall inequalities.} \label{Sec:Gron}

We begin with a Lemma that is a generalization of a result presented in \cite{ChHaYu1}. In \cite{ChHaYu1}, a second order Gronwall inequality for the differential inequality $a\ddot{y}+b\dot{y}+cy +d \leq 0$ with
$a>0$, $b, c, d \in \mathbb{R}$, was proven. Here we show a similar result, but with a continuous function $g(t)$ in the place of the constant $d$.

\begin{lemma} \label{Lemma:Gron1}
Let $y=y(t)$ be a nonnegative $C^2$-function satisfying the second order differential inequality
\begin{equation} \label{Sec-l} a\ddot{y}+b\dot{y}+cy \leq g(t) , \quad y(0)=y_0, \quad \dot{y}(0)=y_1, \quad t>0 ,\end{equation}
where $a>0$, $b,c$ constants, and $g(t)$ is a positive continuous function. Then, we have the estimates for
\\ \\ (i) $b^2-4ac>0 :$

\begin{align} \label{Pos-Dis} y(t) \leq  e^{-\nu_1 t} y_0 + a \frac{e^{-\nu_2 t}-e^{-\nu_1 t}}{\sqrt{b^2-4ac}}\left( y_1+\nu_1 y_0 \right)+ \frac{1}{a}\int_0^t ds \, e^{-\nu_1 (t-s)}
\int_0^s ds' e^{-\nu_2 (s-s')}g(s'), \end{align}
where $\nu_{1}=\frac{b+\sqrt{b^2-4ac}}{2a}$ and $\nu_{2}=\frac{b-\sqrt{b^2-4ac}}{2a}$.
\\ \\(ii) $b^2-4ac \leq 0 :$

\begin{align} \label{Neg-Dis} y(t) \leq e^{-\frac{b}{2a}t} \Bigg[ y_0 +\left( \frac{b}{2a} y_0+y_1 \right)t \Bigg] + \frac{1}{a}\int_0^t ds \int_0^s ds' e^{-\frac{b}{2a}(t-s')}
g(s') .\end{align}
Moreover, if $g(t)$ vanishes asymptotically as $t \to \infty$, i.e. $g(t)\to 0$ as $t \to \infty$. Then, it follows that $y(t)\to 0$ as $t \to \infty$.
\end{lemma}

\begin{proof} The proof follows the ideas in the proof of Lemma 3.1 pg. 34 in \cite{ChHaYu1} with modifications.
\\ \\(i) For $b^2-4ac>0$ we consider the functional
\begin{align*} L(t):=\dot{y}+Ay-\frac{1}{a}\int_0^t \, ds \, e^{-\frac{h}{a}(t-s)}g(s), \end{align*}
where $A=\frac{b-h}{a}$ and $h$ any of the two positive solutions of polynomial $h^2-bh+ac=0$. For this Lyapunov functional
we find easily that $\dot{L}(t)\leq -\frac{h}{a}L(t)$. Then, we have
\begin{equation*} L(t)\leq e^{-\frac{h}{a}t}L(0) = e^{-\frac{h}{a}t} (y_1 +A y_0) ,\end{equation*}
which in turn yields
\begin{equation*} \dot{y}+Ay \leq \frac{1}{a}\int_0^t ds \, e^{-\frac{h}{a}(t-s)}g(s) + e^{-\frac{h}{a}t} (y_1 +A y_0) .\end{equation*}
Now solving in $y(t)$ by integrating factor and taking $h=\frac{b-\sqrt{b^2-4ac}}{2}$ as the solution that optimizes the exponents we derive the bound in \eqref{Pos-Dis}.
\\ \\(ii) For $b^2-4ac \leq 0$ we make the substitution $Y(t):=e^{\frac{b}{2a}t}y(t)$ and derive the following equation for $Y(t)$, i.e.
\begin{equation*} \ddot{Y}(t)- \frac{b^2-4ac}{4a}Y(t) \leq \frac{1}{a}e^{\frac{b}{2a}t}g(t). \end{equation*}
Now since $y(t)$ is nonegative we have $\ddot{Y}(t) \leq \frac{1}{a}e^{\frac{b}{2a}t}g(t)$ which yields after two integrations
\begin{equation*} Y(t) \leq y_0 + \left( \frac{b}{2a}y_0 +y_1 \right)t +\frac{1}{a} \int_0^t ds \, \int_0^s ds' e^{\frac{b}{2a}s'} g(s') .\end{equation*}
The desired inequality \eqref{Neg-Dis} for $y(t)$ follows trivially after we multiply both sides by $e^{-\frac{b}{2a}t}$.
We leave the last part of our proof, i.e. that when the continuous $g(t) \to 0$, as $t \to \infty$, then it follows
that $y(t)\to 0$ for both \eqref{Pos-Dis}-\eqref{Neg-Dis}, in the following autonomous lemma.
\end{proof}

\begin{lemma} \label{Lemma:Integr-vanish} Assume that we have a continuous function $g(t)$ on $[0,\infty)$ s.t. $g(t)\to 0$ as $t \to \infty$. Then the following double integrals (for $a, b, \nu_1, \nu_2 >0 $ and $\nu_1>\nu_2$)
\begin{equation*} \T_1(t):=e^{-\nu_1 t} \int_0^t ds \, e^{(\nu_1-\nu_2)s} \int_0^s ds' \, e^{\nu_2s'}g(s'), \qquad  \T_2(t):=e^{-\frac{b}{2a} t} \int_0^t ds \int_0^s ds' \,
e^{\frac{b}{2a}s'} g(s') ,\end{equation*}
also vanish as $t \to \infty$.
\end{lemma}

\begin{proof} We fix some $\epsilon>0$, and due to the continuity of $g(t)$ we know there
exists some $T_{\epsilon}$ s.t. $g(t)<\epsilon$ for $t \geq T_{\epsilon}$. We can break the double integral
$\T_1(t)$ into 3 terms as follows \begin{align*} \T_1(t)&= e^{-\nu_1 t}\int_0^t ds \ldots  \int_0^s ds' \ldots = \T_{11}(t)+\T_{12}(t)+\T_{13}(t) \\ &=e^{-\nu_1 t}\int_0^{T_{\epsilon}} ds \ldots  \int_0^s ds' \ldots + e^{-\nu_1 t}\int_{T_{\epsilon}}^t ds \ldots  \int_0^{T_{\epsilon}} ds' \ldots + e^{-\nu_1 t}\int_{T_{\epsilon}}^t ds \ldots  \int_{T_{\epsilon}}^s ds' \ldots .\end{align*}
We now compute and integral term separately using the bounds $g(t)\leq \|g\|_{\infty}:=\max \limits_{t \geq 0}g(t)$
on the interval $[0,T_{\epsilon}]$, and $g(t)<\epsilon$ on $[T_{\epsilon},\infty)$, i.e.
\begin{align*} \T_{11}(t) &\leq e^{-\nu_1 t} \frac{\|g\|_{\infty}}{\nu_2}\int_0^{T_{\epsilon}} ds \, e^{(\nu_1-\nu_2)s} (e^{\nu_2 s}-1) \leq
e^{-\nu_1 t} \frac{\|g\|_{\infty}}{\nu_2}\left( \frac{1}{\nu_1}(e^{\nu_1 T_{\epsilon}}-1)-\frac{1}{\nu_1-\nu_2}\left(e^{(\nu_1-\nu_2) T_{\epsilon}}-1\right)\right) ,\\
 \T_{12}(t) &\leq e^{-\nu_1 t} \frac{\|g\|_{\infty}}{\nu_2}\int_{T_{\epsilon}}^t ds \, e^{(\nu_1-\nu_2)s} (e^{\nu_2 T_{\epsilon}}-1) \leq
 e^{-\nu_1 t} \frac{\|g\|_{\infty}}{\nu_2(\nu_1-\nu_2)} \left(e^{(\nu_1-\nu_2) T_{\epsilon}}-1\right)\left(e^{\nu_2 T_{\epsilon}}-1\right) ,\\
 \T_{13}(t) & \leq e^{-\nu_1 t} \frac{\epsilon}{\nu_2}\int_{T_{\epsilon}}^t ds \, e^{(\nu_1-\nu_2)s} (e^{\nu_2 s}-e^{\nu_2 T_{\epsilon}}) \\ & \leq e^{-\nu_1 t} \frac{\epsilon}{\nu_2}\left( \frac{1}{\nu_1}\left(e^{\nu_1 t}-e^{\nu_1 T_{\epsilon}}\right) -\frac{1}{\nu_1-\nu_2}\left(e^{(\nu_1-\nu_2) t} -e^{(\nu_1-\nu_2) T_{\epsilon}}\right)\right) \\ &\leq \frac{\epsilon}{\nu_2}\left( \frac{1}{\nu_1}\left( 1-e^{-\nu_1(t-T_{\epsilon})}\right)-\frac{1}{\nu_1-\nu_2}
 \left( e^{-\nu_2 t} -e^{-\nu_1(t-T_{\epsilon})}e^{-\nu_2 T_{\epsilon}} \right)\right).\end{align*}
 It becomes clear from the $\T_{11}(t), \T_{12}(t), \T_{13}(t)$ estimates that we can choose
 $T_{\epsilon}'=T_{\epsilon}'(\epsilon, T_{\epsilon},\|g\|_\infty)>0$ s.t $\T_{11}(t),\T_{12}(t)<\epsilon$ and also
 $\T_{13}(t)<\frac{2}{\nu_1 \nu_2}\epsilon$ for $t>T_{\epsilon}'$, i.e.
 $\T_{1}<\left( 2+\frac{2}{\nu_1 \nu_2}\right) \epsilon$ for $t>T_{\epsilon}'$. The proof for the term $\T_2(t)$
 is similar and therefore omitted.
\end{proof}

It is useful to derive uniform (in time) bounds based on the above
estimates, which also involve the integrability of $g(t)$ in time, i.e.
a uniform bound with respect to the integral $\int_0^{\infty} g(t)\, dt <\infty$.
We show the uniform bounds in the following corollary.

\begin{corollary} \label{Corollary:Gron1} Let $y=y(t)$ be a nonnegative $C^2$-function satisfying the second order
differential inequality \eqref{Sec-l}, and $g(t)$ is an integrable function. Then we have the
following uniform bounds for $y(t)$:
\\(i) $b^2-4ac > 0 $ :
\begin{align*} y(t)\leq y_0 +\frac{a}{\sqrt{b^2-4ac}}(|y_1| +\nu_1 y_0) +\frac{1}{2 \sqrt{ac}}\int_0^{\infty}ds \, g(s) \end{align*}
\\(ii) $b^2-4ac \leq 0 $ :

if $y_1 \geq 0$ :
\begin{align*} y(t)\leq e^{-\frac{y_0}{\frac{b}{2a} y_0 +y_1}} \left(1+\frac{2a}{b} \right) y_0 +\frac{1}{b}\int_0^{\infty}ds \, g(s) ,\end{align*}

if $y_1<0$ :
\begin{align*} y(t)\leq y_0 +\frac{1}{b}\int_0^{\infty}ds \, g(s) .\end{align*}
\end{corollary}

\begin{proof}

It is an elementary exercise in calculus to give a uniform bound for the non integral terms that appear in \eqref{Pos-Dis} and \eqref{Neg-Dis} in Lemma 1.
The most important terms to bound are the integral terms \begin{equation*}\int_0^t ds \, e^{-\nu_1 (t-s)} \int_0^s ds' e^{-\nu_2 (s-s')}g(s') \quad \text{and} \quad \int_0^t ds \int_0^s ds' e^{-\frac{b}{2a}(t-s')} g(s') .\end{equation*} We have the following H\"older inequality for the first integral
\begin{align*}  \int_0^t ds \, e^{-\nu_1 (t-s)} \int_0^s ds' e^{-\nu_2 (s-s')}g(s') \leq \left( \int_0^t ds \, e^{-2 \nu_1 (t-s)} \right)^{1/2}
\left( \int_0^t ds \, \Big| \int_0^s ds' \, e^{-\nu_2(s-s')} g(s')\Big| ^2 \right)^{1/2} .\end{align*}
The first term computed by direct calculation is $\sqrt{\frac{1}{2 \nu_1}(1-e^{-2\nu_1 t})}\leq \sqrt{\frac{1}{2 \nu_1}}$. The second term is bounded
by the integral
$\left( \int_0^{\infty} ds \, \Big| \int_0^s ds' \, e^{-\nu_2(s-s')} g(s') \Big|^2 \right)^{1/2}$ which is non other than the $L^2$ norm of the convolution
of the function $f(t)=e^{-\nu_2 t}$ with $g(t)$, i.e. $\|f \ast g\|_{2}$. Now according to Young's inequality for convolutions
we have $\|f \ast g\|_r \leq \|f \|_p \| g\|_q $ for $\frac{1}{p}+\frac{1}{q}=1+\frac{1}{r}$ and $1 \leq p, q, r \leq \infty$, and if we take $p=2$, $q=1$, and $r=2$
then we have that the last term is bounded by $\frac{1}{\sqrt{2 \nu_2}} \|g\|_1$ so we  have
\begin{align*} \int_0^t ds \, e^{-\nu_1 (t-s)} \int_0^s ds' e^{-\nu_2 (s-s')}g(s')
\leq \frac{1}{2\sqrt{\nu_1 \nu_2}}\|g \|_1 =\frac{1}{2}\sqrt{\frac{a}{c}}\|g \|_1.\end{align*}
In similar fashion we can show that
\begin{align*} \int_0^t ds \int_0^s ds' e^{-\frac{b}{2a}(t-s')} g(s') \leq \frac{a}{b} \|g\|_1 .\end{align*}
\end{proof}

We now proceed to prove a different Lemma for a Gronwall inequality which is
more tailored to the structure of \eqref{BasIneq2}, in the way that Lemma 1 was more specific to the structure of \eqref{BasIneq1}.
This is a second order differential inequality that includes a convolution term and an exponential
decay term in the RHS.
\begin{equation} \label{Sec-nl} \ddot{y}+a\dot{y}+b y \leq c \int_0^t e^{-\nu (t-s)}y(s)\, ds
+de^{-\nu t}, \qquad t\geq 0, \quad a,b,c,d \geq 0 .\end{equation}

\begin{lemma} \label{Lemma:Gron2} Let $y=y(t)$ be a nonnegative $C^2$-function
satisfying the differential inequality \eqref{Sec-nl},
with initial conditions $y(0)=y_0$ and $\dot{y}(0)=y_1$. Assume also that we have the set
\begin{equation} \label{D} \D:=\left\{\mu >0 : \mu^3-(a+\nu)\mu^2 +(b+a \nu)\mu \leq d_* ,
\quad \mu <\min \{ a, \nu \} \right\}\end{equation}
where $d_*=b \nu -c$.
If $d_*>0$ then the set $\D $ is nonempty, and if we define $\mu_*:=\sup \D$ we have the following estimates
\begin{equation*} y(t) \leq e^{-(a-\mu_*) t}y_0 +\frac{1}{a-2\mu_*}
\left(e^{-\mu_* t}-e^{-(a-\mu_*) t}\right)\left(y_1+(a-\mu_*) y_0 +\frac{d}{\nu -\mu_*}\right),
\quad \text{if} \quad \mu_* < \frac{a}{2}\end{equation*}
and
\begin{equation*} y(t)\leq e^{-\frac{a}{2}t}y_0 +\left(y_1 +\frac{a}{2}y_0 +
\frac{d}{\nu-\frac{a}{2}} \right)t e^{-\frac{a}{2}t} , \quad \text{if}\quad \mu_* \geq \frac{a}{2} .\end{equation*}
\end{lemma}

\begin{proof}
The approach that we take is to consider the functional
\begin{equation*} L(t):=\dot{y}(t)+d_1 y(t) +d_2 \int_0^t
e^{-\nu (t-s)}y(s)\, ds +d_3 e^{-\nu t},\end{equation*} and try
to find values for the coefficients $d_1, d_2, d_3$ and some $\mu >0$ so that $L(t)$ decays exponentially
with rate $\mu>0$ i.e. $\dot{L}(t)\leq -\mu L(t)$.  If we can find such
$\mu>0$, then (keeping in mind that $y(t)\geq 0$) we have
\begin{align*} \dot{y}(t)+d_1 y(t) \leq e^{-\mu t}L(0)=e^{-\mu t}(y_1 +d_1 y_0 +d_3).\end{align*}
Solving this first order inequality by using integrating factor, we get that
\begin{equation*} y(t)\leq e^{-d_1 t}y_0 +\frac{1}{d_1-\mu}
\left(e^{-\mu t}-e^{-d_1 t}\right)(y_1+d_1 y_0 +d_3)
\quad \text{for} \quad d_1 \neq \mu\end{equation*}
and that \begin{equation*} y(t)\leq e^{-d_1 t}y_0+(y_1+d_1 y_0 +d_3)t e^{-d_1 t}  \quad
\text{for} \quad d_1=\mu. \end{equation*}
Of course, if there are many different choices for $d_1, d_2, d_3, \mu >0$ the ideal scenario is
to pick those that maximize the decrease rates for both $d_1$ and $\mu$.
In order to compute $\dot{L}(t)$ we make use the fact that \begin{equation*}\frac{d}{dt}\int_0^te^{-\nu (t-s)}y(s)\, ds
=y(t)-\nu \int_0^te^{-\nu (t-s)}y(s)\, ds . \end{equation*} Let us now
compute $\dot{L}(t)+\mu L(t)$ and set it $\leq 0$ to derive the system
of all conditions that the parameters need to satisfy so that we have dissipation. Indeed, we have
\begin{align*} \dot{L}(t)+ \mu L(t)=\ddot{y}+(d_1 +\mu)\dot{y}+(d_2+d_1 \mu)y
+d_2(\mu-\nu)\int_0^te^{-\nu(t-s)}y(s)\, ds +d_3(\mu-\nu)e^{-\nu t} \leq 0 .\end{align*}
We use inequality \eqref{Sec-nl} to identify the set of conditions
for which $\dot{L}(t)+\mu L(t)\leq 0$. This yields the
following system of equalities \& inequalities that need to be satisfied all at the same time, i.e.
\begin{subequations} \label{Syst}  \begin{align}  \label{Syst1} d_1+\mu &=a
\\  \label{Syst2} d_2 + d_1 \mu  & \leq  b
\\  \label{Syst3} d_2(\mu-\nu) & \leq -c
\\  \label{Syst4} d_3(\mu-\nu)  & \leq -d
\end{align}
\end{subequations}
We now have to check the validity of conditions in system
\eqref{Syst} and then pick the parameters so we get the
maximum decay rates. We first choose $\mu$ such that
\begin{equation} \label{Cond1} 0<\mu <\min \{ a, \nu \}.\end{equation}
We see that \eqref{Syst4} is a condition independent of \eqref{Syst1}-\eqref{Syst3},
since the $d_3$ term is not involved in any of \eqref{Syst1}-\eqref{Syst3}, i.e.
we can make the following choice for $d_3$ after we find admissible $\mu$ that solve the system
\begin{equation*}  d_3 = \frac{d}{\nu -\mu}. \end{equation*}
We can now solve \eqref{Syst1} in $d_1$ ($d_1=a -\mu$) and substitute into \eqref{Syst2} to get
\begin{equation} \label{cond-d2} d_2 \leq b -d_1 \mu = \mu^2 -a\mu +b .\end{equation}
What is left to do, is to find values of $\mu>0$ s.t. \eqref{cond-d2} and \eqref{Syst3} are simultaneously satisfied.
This is true, if we can find $\mu$ that satisfies the inequality
\begin{equation} \label{Cond2} \mu^2 -a\mu +b \geq \frac{c}{\nu -\mu} ,\end{equation}
then by \eqref{Syst3} we can choose any $d_2$ such that
\begin{equation*} \mu^2 -a\mu +b  \geq d_2 \geq \frac{c}{\nu -\mu} .\end{equation*}
The two conditions \eqref{Cond1} and  \eqref{Cond2} are sufficient conditions for the system \eqref{Syst1}-\eqref{Syst4} to have admissible solutions. Practically, \eqref{Cond2} leads to the cubic
inequality contained in the definition of set $\D$. Notice that when $b \nu >c$, then some $\mu>0$ that solves both \eqref{Cond1} and  \eqref{Cond2} exists.

The last step, after we make sure that there is some $\mu>0$ so that \eqref{Cond1} and \eqref{Cond2} are satisfied, is to pick the one $\mu$ that maximizes the decay rate.
For that, we define the set $\D$ like we did in \eqref{D} and the supremum $\mu_*$ of that set. If $\mu_*\geq \frac{a}{2}$ then the optimal decay rate is $\lambda=\frac{a}{2}$,
otherwise, if $\mu_*<\frac{a}{2}$ then the optimal rate is $\lambda=\mu_*$.
\end{proof}

The following uniform estimate can be shown with the help of Lemma \ref{Lemma:Gron2}

\begin{corollary} \label{Corollary:Gron2} Let $y=y(t)$ be a nonnegative $C^2$-function
satisfying the differential inequality \eqref{Sec-nl},
with initial conditions $y(0)=y_0$ and $\dot{y}(0)=y_1$ and that condition
$b \nu >c$ is satisfied. Then we have the following bound for $y(t)$
\begin{equation*} y(t)\leq e^{-a t}y_0 +\frac{1}{a}(1-e^{-a t})
\left( y_1 +a y_0 +\frac{d}{\nu}\right), \qquad t \geq 0.\end{equation*} In particular, we have the uniform bound \begin{equation*} y(t)< y_0 +\frac{|y_1|}{a}+\frac{d}{a \nu}, \qquad t \geq 0.\end{equation*}
\end{corollary}

\begin{proof} The proof employs the estimates for $L(t)$ that we proved in Lemma \ref{Lemma:Gron2}, only now we restrict to the case
of a functional $L(t)$ with $\mu=0$ (no decay rate), so in effect we use the functional \begin{equation*} L(t)=\dot{y}(t)+a y(t)+\frac{c}{\nu}
\int_0^te^{-\nu(t-s)}y(s)\, ds +\frac{d}{\nu}e^{-\nu t} .\end{equation*}
\end{proof}

\section{Proof of Main Theorems} \label{Sec:Proofs}

We are now in position to prove the main theorems of this paper.

\begin{proof}[\textbf{Proof of Theorem \ref{Theorem1}}] The proof essentially uses Proposition \ref{Prop:Vel-diam}, the Gronwall lemmas proved in Section \ref{Sec:Gron}, and a continuity argument.

\textbf{Step A}: Let $T>0$ be an arbitrary number so that $\A(v(t))>\delta_0$ for all $t \in [0,T)$ (we know that such a $T>0$ exists by continuity of $\A(v)$). Then, we have that \eqref{BasIneq1} holds
on that interval $[0,T)$, and subsequently we may use Lemma \ref{Lemma:Gron1} with
\begin{equation*} a=1, \qquad b=\frac{\gamma}{\chi}, \qquad c=\frac{2k}{\chi}\psi_m \delta_0 , \qquad g(t)=\frac{4}{\chi^2}\left(D(s)^2 +D(v)^2 \max \limits_{1 \leq i \leq N}|s_i|^2 \right).\end{equation*}
In order to control $g(t)$ in terms of $\S(t)$ we have
\begin{equation*} D(s)^2=\max \limits_{i,j}|s_i -s_j|^2 \leq (|s_l|+|s_m|)^2 \leq 2(|s_l|^2+|s_m|^2)\leq 2 N \S(t) ,\end{equation*}
where $l,m$ are the indices that realize the $\max \limits_{i,j}|s_i -s_j|^2$, and also
\begin{equation*} \max \limits_{i} |s_i|^2 \leq N \S(t), \quad \text{which both lead to} \quad g(t)\leq \frac{16 N}{\chi^2}\S(t) .\end{equation*}
Now if we integrate $g(t)$ and use Proposition \ref{Prop:Ha}, we have that
\begin{equation*}  \int_0^\infty g(t) \, dt \leq \frac{16 N}{\chi^2} \int_0^\infty \S(t) \, dt\leq \frac{8 N k}{\chi \gamma}\left( \frac{\chi}{2}\E(0)+\frac{1}{k}\S(0)\right) . \end{equation*}
We have that for all $t \in [0,T)$ the estimates \eqref{Pos-Dis}--\eqref{Neg-Dis} in Lemma \ref{Lemma:Gron1} hold.

\textbf{Step B}: We now define the set
\begin{equation*} \S:=\{ T >0 : \A(v(t))>\delta_0, \quad \forall t \in [0,T) \}, \qquad T_*:=\sup \S .\end{equation*}
We have shown in Step A due to the continuity property that the set $\S$ is non-empty. We claim that $T_*=\infty$ so that the estimates above hold true for all $t \geq 0$. Indeed, assume that $T_*<\infty$, then we
have that $\lim \limits_{t \to T_*^-}\A(v(t))=\delta_0$ or equivalently $\lim \limits_{t \to T_*^-}D(v(t))^2=2-2\delta_0$. Now from
Corollary \ref{Corollary:Gron1} we have that
\begin{equation*} D(v(t))^2 \leq 2 \C_0 <2-2\delta_0 , \qquad \forall t \in [0, T_*) .\end{equation*}
Taking the limit $\lim \limits_{t \to T_*^-} D(v(t))^2$ leads to the contradiction $2-2\delta_0<2-2\delta_0$. Therefore $T_*=\infty$, and the estimates from Step A for $D(v)$ hold.
Finally, the vanishing of velocity diameters follows from Lemma \ref{Lemma:Integr-vanish}.
\end{proof}

\begin{proof}[\textbf{Proof of Theorem \ref{Theorem2}}] The proof follows the lines of the proof of Theorem \ref{Theorem1}.

\textbf{Step A}: Let $T>0$ be an arbitrary number so that $\A(v(t))>\delta_0$ for all $t \in [0,T)$ (we know that such a $T>0$ exists by continuity of $\A(v)$). Then, we have that \eqref{BasIneq2} holds
on that interval $[0,T)$, and subsequently we may use Lemma \ref{Lemma:Gron2} with
\begin{equation*} a=\frac{\gamma}{\chi}, \qquad b=\frac{2k}{\chi}\psi_m \delta_0, \qquad c=\frac{12k^2 \psi_M}{\gamma \chi} ,\qquad d=\frac{4}{\chi^2}\left( D(s_0)^2 +2 \max \limits_{1 \leq i \leq N}|s_{i0}|^2\right)\end{equation*}
where the set $\D$ in Lemma \ref{Lemma:Gron2} takes the form of \eqref{D_set}. Since $(H1)$ holds, we have that $\D$ is non-empty and $\mu_*>0$. Then for all $t \in [0,T)$ we have the estimates
\begin{align*} D(v)^2 \leq \left\{ \begin{array}{ll} e^{-d_1 t}D(v_0)^2 +\frac{1}{d_1-\mu_*}
\left(e^{-\mu_* t}-e^{-d_1 t}\right)\left(2 D(v_0)|\dot{D}(v_0)|+
d_1 D(v_0)^2 +\frac{d}{d_1}\right) , \quad \text{for} \quad \mu_* < \frac{\gamma}{2\chi} ,
  \\ \\ e^{-\frac{\gamma}{2\chi}t} D(v_0)^2 +\left(2 D(v_0) |\dot{D}(v_0)| +\frac{\gamma}{2\chi}D(v_0)^2 +\frac{2d \chi}{\gamma} \right)t e^{-\frac{\gamma}{2\chi}t} ,\quad \text{for} \quad  \mu_* \geq \frac{\gamma}{2 \chi} , \end{array} \right. \end{align*} with $d_1=\frac{\gamma}{\chi}-\mu_*$.
\\ \\
\textbf{Step B}: We follow the exact same steps in Step B, as in the proof of Theorem \ref{Theorem1}.
\end{proof}


\quad (Ioannis Markou)
\\
\textsc{Institute of Applied and Computational Mathematics
(IACM-FORTH)}
\\
\textsc{N. Plastira 100, Vassilika Vouton GR - 700 13, Heraklion,
Crete, Greece}

\quad E-mail address: \textbf{ioamarkou@iacm.forth.gr}

\end{document}